\documentclass[12pt]{amsart}
\usepackage[colorlinks]{hyperref}
\usepackage{amssymb}
\usepackage{amscd}
\usepackage{fullpage}
\usepackage{amsmath}

\numberwithin{equation}{section}
\theoremstyle{plain}
\begin{document}
\newtheorem{theorem}{Theorem}[section]
\newtheorem{conjecture}[theorem]{Conjecture}
\newtheorem{corollary}[theorem]{Corollary}
\newtheorem{definition}[theorem]{Definition}
\newtheorem{lemma}[theorem]{Lemma}
\newtheorem{proposition}[theorem]{Proposition}
\newtheorem{remark}[theorem]{Remark}
\newtheorem{example}[theorem]{Example}


   \newtheorem{thm}{Theorem}[subsection]
   \newtheorem{cor}[thm]{Corollary}
   \newtheorem{lem}[thm]{Lemma}
   \newtheorem{prop}[thm]{Proposition}
   \newtheorem{defi}[thm]{Definition}
   \newtheorem{rem}[thm]{Remark}
   \newtheorem{exa}[thm]{Example}
   
\title[Levin-Cochran-Lee inequalities and best constants on homogeneous groups]{Levin-Cochran-Lee inequalities and best constants on homogeneous groups}

\author{Michael Ruzhansky and Markos Fisseha Yimer}

\address{Department of Mathematics: Analysis, Logic and Discrete Mathematics\\
 Ghent University, Belgium\\ and\\ School of Mathematical Sciences\\Queen Mary University of London\\ United Kingdom\\ {\em E-mail: michael.ruzhansky@ugent.be} }
\address{Department of Mathematics\\ Addis Ababa University, 1176 Addis Ababa, Ethiopia {\em E-mail: markos.fisseha@aau.edu.et} }



\begin{abstract}
In this paper, we apply a direct method instead of a limit approach, for proving the Levin-Cochran-Lee inequalities. First, we state and prove Levin-Cochran-Lee type inequalities on a homogeneous group $\mathbb{G}$ with parameters $0<p\leq q<\infty$. Furthermore, for the case $p=q$, we prove the sharp inequalities with power weights and derive some other new inequalities.
\end{abstract}

\maketitle

\section{Introduction}\label{sec1}
G. H. Hardy stated and proved in his 1925 paper \cite{H} the following inequality:
\begin{align}\label{H.In.Eq.1.1}
\int_0^\infty \left(\dfrac{1}{x}\int_0^x f(y)\,\mathrm{d}y\right)^p\,\mathrm{d}x\leq \left(\dfrac{p}{p-1}\right)^p \int_0^\infty f^p(x)\,\mathrm{d}x,
\end{align}
for $p>1$, where $f\geq 0$ is a non-negative measurable function on $(0,\infty)$. The inequality is usually called the classical Hardy inequality. Three years later, in 1928, he proved the following generalization of \eqref{H.In.Eq.1.1}:
\begin{align}\label{H.In.Eq.1.2}
\int_0^\infty \left(\dfrac{1}{x}\int_0^x f(y)\,\mathrm{d}y\right)^p x^\alpha\,\mathrm{d}x\leq \left(\dfrac{p}{p-1-\alpha}\right)^p \int_0^\infty f^p(x) x^\alpha\,\mathrm{d}x,
\end{align}
whenever $p\geq 1$ and $\alpha<p-1$. Moreover, the constant $\left(\dfrac{p}{p-1-\alpha}\right)^p$ is sharp (see \cite{H1}). Since then a lot of generalizations and complementary results have been published see e.g. \cite{HLP, KPDP, KMP2, KMP1, KPS, OK} and many references therein. 
Let us mention one of the results which is also considered as the limiting case of the classical Hardy inequality. We note that by replacing $f$ with $f^\frac{1}{p}$ in \eqref{H.In.Eq.1.1} and letting $p\to \infty$, we obtain the following P$\acute{\text{o}}$lya-Knopp inequality
\begin{align}\label{H.In.Eq.1.3}
\int_0^\infty \exp{\left(\dfrac{1}{x}\int_0^x f(t)\,\mathrm{d}t\right)}\, \mathrm{d}x\leq e \int_0^\infty f(x)\, \mathrm{d}x.
\end{align}
The constant $e$ in \eqref{H.In.Eq.1.3} is sharp. Moreover, by making a similar limiting procedure in \eqref{H.In.Eq.1.2} we obtain the following weighted version of \eqref{H.In.Eq.1.3}:
\begin{align*}
\int_0^\infty \exp{\left(\dfrac{1}{x}\int_0^x f(t)\,\mathrm{d}t\right)} x^\alpha\, \mathrm{d}x\leq e^{(1+\alpha)} \int_0^\infty f(x) x^\alpha\, \mathrm{d}x,
\end{align*}
for $\alpha>-1$ and all non-negative measurable functions $f\geq 0$ on $(0,\infty)$. These inequalities have been generalized, complemented and discussed in several publications, see e.g. \cite{CPP, AJL, CL, HKK, MRDS, Y, YPA} and the references given there. In particular, the following exponential weighted inequality was proved by Cochran and Lee in 1984 (see \cite{CL}):
\begin{theorem}
Let $\beta, \gamma$ be real numbers with $\beta>0$. If $\int_0^\infty x^\gamma f(x)\mathrm{d}x<\infty$, then 
\begin{align}\label{H.In.Equ.1.4}
\int_0^\infty x^\gamma\exp{\left(\beta x^{-\beta}\int_0^x t^{\beta-1}\log f(t)\mathrm{d}t\right)}\mathrm{d}x\leq \exp{\left(\dfrac{\gamma+1}{\beta}\right)} \int_0^\infty x^\gamma f(x)\mathrm{d}x,
\end{align}
holds for all positive functions $f$. Moreover, the constant $\exp{\left(\frac{\gamma+1}{\beta}\right)}$ is sharp.
\end{theorem}
 Here we want to note that the inequality \eqref{H.In.Equ.1.4} was proved earlier in 1938 by Levin in his paper \cite{L} which was written in the Russian language, and then rediscovered by J. A. Cochran and C.-S. Lee.  This is the main reason why the inequality \eqref{H.In.Equ.1.4} has got the name of Levin-Cochran-Lee type inequality.

\begin{remark}
Recently, Yimer et al. \cite{YPA} studied a general $n$-dimensional analogue of \eqref{H.In.Equ.1.4} on $I_n:=[0,b_1)\times\cdots\times[0,b_n)\subseteq \mathbb{R}_+^n$, for $0<b_i\leq \infty$, ($i=1,...,n$) involving general weight functions with parameters $0<p\leq q<\infty$. Moreover, estimates of the sharp constant were also discussed (see also \cite{MRASBT}).
\end{remark}
For the purpose of this paper we need to mention the following characterization of the general weight functions $\phi$ and $\psi$ on homogeneous groups for a maximal integral weighted Hardy inequality to hold (see \cite[Theorem 5.4.1]{MRDS} and \cite[Theorem 5.1]{MRDSNY}):

\begin{theorem} [Maximal integral weighted Hardy inequality]\label{H.In.T.1.1}
Let $\mathbb{G}$ be a homogeneous group with the homogeneous dimension $Q$ equipped with a homogeneous quasi-norm $|\cdot|$. Let $\phi$ and $\psi$ be positive functions defined on $\mathbb{G}$. Then there exists a constant $C>0$ such that 
\begin{align*}
\int_{\mathbb{G}} \phi(x)\exp{\left(\dfrac{1}{|B(0,{|x|})|}\int_{B(0,{|x|})}\ln f(y)\mathrm{d}y\right)} \mathrm{d}x
    \leq C \int_{\mathbb{G}} \psi(x) f(x)\mathrm{d}x
\end{align*}
holds for all positive $f$ defined on $\mathbb{G}$ if and only if 
\begin{align*}
A:=\sup\limits_{R>0}\, R^Q \int_{|x|\geq R} \dfrac{\phi(x)}{|x|^{2Q}} \exp{\left(\dfrac{1}{|B(0,{|x|})|}\int_{B(0,{|x|})}\ln \left(\dfrac{1}{\psi(y)}\right)\mathrm{d}y\right)}\,   \mathrm{d}x<\infty.
\end{align*}
\end{theorem}
\begin{remark}
Inequalities of the type of those in Theorem \ref{H.In.T.1.1} where $\mathbb{G}$ is replaced by $I_n=[0,b_1)\times\cdots[0,b_n)\subseteq \mathbb{R}_+^n$, with $0<b_i\leq \infty,\,(i=1,...,n)$ and the means are considered over a hyperrectangle, were studied in \cite{Y} for the multidimensional case with $0<p\leq q<\infty$. Moreover, estimates of the sharp constant were also discussed (see also \cite{JH, AW}  and the references therein).
\end{remark}

The main purpose of this paper is to study the Levin-Cochran-Lee type inequality on a homogeneous group $\mathbb{G}$ equipped with a quasi-norm $|\cdot|$ for the case $0<p\leq q<\infty$ using a direct method. The paper is organized as follows: In Section 2 we give some basics on homogeneous groups. In Section 3 we state and prove the general main results (see Theorems \ref{H.T.4.1} and \ref{H.T.5.1}) and we discuss also the sharp inequalities with power weight functions for the case $p=q$ (see Theorems \ref{H.T.4.2} and \ref{H.T.5.2}). Finally, Section 4 is reserved for some concluding remarks.  
\section{Preliminaries}
In this section, we recall the basics of homogeneous groups. For further reading on homogeneous groups and other inequalities on homogeneous groups, we refer to the monographs \cite{FR, FH, MRDS} and references therein. 
\subsection{ Basics on homogeneous Lie groups}

A Lie group $\mathbb{G}$ (identified with $(\mathbb{R}^N, \circ)$) is called a homogeneous Lie group if it is equipped with a dilation mapping
\begin{align*}
D_\lambda: \mathbb{R}^N\to \mathbb{R}^N, \, \lambda>0,
\end{align*}
defined as
\begin{align*}
D_\lambda(x)=(\lambda^{v_1}x_1, \lambda^{v_2}x_2,...,\lambda^{v_N}x_N), v_1,v_2,...,v_N>0,
\end{align*}
which is an automorphism of the group $\mathbb{G}$ for each $\lambda>0$. Here and in the sequel, we will denote the image of $x\in \mathbb{G}$ under $D_\lambda$ by $\lambda(x)$ or, simply $\lambda x$. The homogeneous dimension $Q$ of a homogeneous Lie group $\mathbb{G}$ is defined by 
\begin{align*}
Q=v_1+v_2+\cdots+v_N.
\end{align*}
It is well known that a homogeneous group is necessarily nilpotent and unimodular. There are different particular examples of homogeneous groups such as the Euclidean space $\mathbb{R}^n$ (in which case $Q=n$), the Heisenberg group, as well as general stratified groups (homogeneous Carnot groups) and graded groups.

The Haar measure $\mathrm{d}x$ on $\mathbb{G}$ is nothing but the Lebesgue measure on $\mathbb{R}^N$. \\
Let us denote the volume of a measurable set $\omega\subseteq \mathbb{G}$ by $|\omega|$. Then we have the following consequences: for $\lambda>0$
\begin{align*}
|D_\lambda(\omega)|=\lambda^Q|\omega| \text{ and } \int_{\mathbb{G}} f(\lambda x)\mathrm{d}x=\lambda^{-Q}\int_{\mathbb{G}} f(x)\mathrm{d}x.
\end{align*}
A quasi-norm on $\mathbb{G}$ is any continuous non-negative function $|\cdot|: \mathbb{G}\to [0, \infty)$ satisfying the following conditions:
\begin{enumerate}
    \item[(i)] $|x|=|x^{-1}|$ for all $x\in \mathbb{G}$
     \item[(ii)] $|\lambda x|=\lambda|x|$ for all $x\in \mathbb{G}$ and $\lambda>0$
      \item[(iii)] $|x|=0 \Longleftrightarrow x=0.$
\end{enumerate}
Before we finish this section, we need to mention the following polar decomposition on a homogeneous group $\mathbb{G}$ since it plays an important role in the proofs of our main results.
\begin{theorem} (c.f. e.g. \cite[Proposition 1.2.10]{MRDS})
Let 
\begin{align*}
\mathfrak{S}=\{x\in \mathbb{G}: |x|=1\}\subset \mathbb{G}
\end{align*}
be the unit sphere with respect to the quasi-norm $|\cdot|$. Then there is a unique Radon measure $\sigma$ on $\mathfrak{S}$ such that for all $f\in L^1(\mathbb{G})$, 
\begin{align}\label{H.I.Eq.3}
    \int_{\mathbb{G}} f(x)\mathrm{d}x= \int_0^\infty \int_{\mathfrak{S}} f(ry)r^{Q-1} \mathrm{d}\sigma(y)\mathrm{d}r.
\end{align}
\end{theorem}
Here and in the sequel, we use the following notations. The letters $u$ and $v$ will be the weights on homogeneous groups $\mathbb{G}$. A quasi-ball in the homogeneous group $\mathbb{G}$ with radius $|x|$, $x\in \mathbb{G}$, and centred at the origin will be denoted by $B(0, |x|)$. We denote the surface measure of the unit sphere $\mathfrak{S}$ in $\mathbb{G}$ by $|\mathfrak{S}|$.

 The Haar measure of the quasi-ball $B(0, |x|)$ denoted by $|B(0, |x|)|$, can be calculated by using polar decomposition \eqref{H.I.Eq.3} as
\begin{align*}
    |B(0, |x|)|&=\int_{B(0, |x|)} \mathrm{d}y=\int_{0}^{|x|} r^{Q-1}\left(\int_{\mathfrak{S}} \mathrm{d}\sigma(t)\right)\mathrm{d}r=\dfrac{|\mathfrak{S}|}{Q}|x|^Q.
\end{align*}
\section{Main results}
We begin this section with the following main theorem involving general weight functions.
\begin{theorem}\label{H.T.4.1}
Let $\mathbb{G}$ be a homogeneous group with the homogeneous dimension $Q$ equipped with a quasi-norm $|\cdot|$, and let $0<p\leq q<\infty$ and $\beta>0$. Suppose that $u$ and $v$ are positive functions on $\mathbb{G}$. Then 
\begin{align}\label{H.Eq.4.1}
\begin{split}
    \left(\int_{\mathbb{G}} \left[\exp{\left(\dfrac{\beta}{|B(0, |x|)|^{\beta}} \int_{B(0, |x|)}|B(0, |y|)|^{\beta-1}\ln f(y)\mathrm{d}y\right)}\right]^q u(x)\mathrm{d}x\right)^{\frac{1}{q}}\\
    \leq C \left(\int_{\mathbb{G}} f^p(x)v(x)\mathrm{d}x\right)^{\frac{1}{p}}
    \end{split}
\end{align}
holds for some finite $C$ and for all positive $f$ if and only if for some $\alpha>0$
\begin{align}\label{H.Eq.4.2}
A(\alpha):=\sup\limits_{R>0}\,  R^\frac{{Q(\beta-1)+\alpha}}{p}\left(\int_{\mathbb{G}\setminus B(0, R)}   |x|^{-(Q\beta+\alpha)\frac{q}{p}} w(x)\mathrm{d}x\right)^\frac{1}{q}<\infty,
\end{align}
where 
\begin{align}\label{H.Eq.4.2.1}
w(x)=u(x)\left[\exp{\left(\dfrac{\beta}{|B(0, |x|)|^{\beta}}\int_{B(0, |x|)}|B(0, |y|)|^{\beta-1}\ln v^{-1}(y)\mathrm{d}y\right)}\right]^\frac{q}{p}.
\end{align}
Moreover, the best constant $C$ in \eqref{H.Eq.4.1} satisfies
\begin{align}\label{H.Eq.4.3}
\begin{split}
\left(\dfrac{1}{|B(0, 1)|}\right)^{1/p}\sup\limits_{\alpha>0} \left(1+\dfrac{\exp{\left(-\dfrac{\beta Q+\alpha}{\beta Q}\right)}}{\alpha/Q+(\beta-1)} \right)^{-1/p} A(\alpha)\\
\leq
 C\leq \left(\dfrac{\beta}{|B(0, 1)|}\right)^{1/p}\inf\limits_{\alpha>0}  \exp{\left(\dfrac{\alpha}{\beta p Q}\right)} A(\alpha). 
\end{split}
\end{align}
Furthermore, if $A(\alpha)<\infty$ for some $\alpha>0$, then $A(\alpha)<\infty$ for all $\alpha>0$.
\end{theorem}
\begin{proof}
{\it Sufficiency.} Let $g(x)=f^p(x)v(x)$.  Then the inequality \eqref{H.Eq.4.1} is equivalent to 
\begin{align}\label{H.Eq.4.4}
\begin{split}
   \left(\int_{\mathbb{G}} \left[\exp{\left(\dfrac{\beta}{|B(0, |x|)|^{\beta}} \int_{B(0, |x|)}|B(0, |y|)|^{\beta-1}\ln g(y)\mathrm{d}y\right)}\right]^\frac{q}{p} w(x)\mathrm{d}x\right)^{\frac{1}{q}}\\
    \leq C \left(\int_{\mathbb{G}} g(x)\mathrm{d}x\right)^{\frac{1}{p}},
    \end{split}
\end{align}
where $w(x)$ is defined by \eqref{H.Eq.4.2.1}. Clearly,
\begin{align}\label{F.H.Equ.3.6a}
\int_{B(0, |x|)}|B(0, |y|)|^{\beta-1}\ln g(y)\mathrm{d}y=\left(\dfrac{|\mathfrak{S}|}{Q}\right)^{\beta-1}\int_0^{|x|}\int_{\mathfrak{S}} r^{\beta Q-1} \ln g(r\xi)\mathrm{d}\sigma(\xi)\mathrm{d}r.
\end{align}
Let $r=|x|s$. Then, from \eqref{F.H.Equ.3.6a}, we have that
\begin{align*}
\int_{B(0, |x|)}|B(0, |y|)|^{\beta-1}\ln g(y)\mathrm{d}y=\\\dfrac{|B(0, |x|)|^\beta}{|B(0,1)|} \int_0^{1}\int_{\mathfrak{S}} s^{Q-1} |s\xi|^{Q(\beta-1)} \ln g(|x|s\xi)\mathrm{d}\sigma(\xi)\mathrm{d}s=\\
\dfrac{|B(0, |x|)|^\beta}{|B(0,1)|} \int_{B(0, 1)} |\xi|^{Q(\beta-1)}\ln g(|x|\xi) \mathrm{d}\xi,
\end{align*}
which implies that
\begin{align}\label{H.Eq.4.5}
\begin{split}
\exp{\left(\dfrac{\beta}{|B(0, |x|)|^{\beta}} \int_{B(0, |x|)}|B(0, |y|)|^{\beta-1}\ln g(y)\mathrm{d}y\right)}\\=\exp{\left(\dfrac{\beta}{|B(0, 1)|}\int_{B(0, 1)} |\xi|^{Q(\beta-1)}\ln g(|x|\xi) \mathrm{d}\xi\right)}.
\end{split}
\end{align}
For any $\alpha>0$, we have that
\begin{align*}
\int_{B(0, 1)}|\xi|^{Q(\beta-1)}\ln (|\xi|^\alpha) \mathrm{d}\xi&= \int_0^{1}\int_{\mathfrak{S}} r^{\beta Q-1} \ln r^\alpha \mathrm{d}\sigma(y)\mathrm{d}r\\
&=\dfrac{\alpha}{\beta}|B(0,1)|\int_0^1 \ln r\,\mathrm{d}(r^{\beta Q})=-\dfrac{\alpha}{\beta}|B(0,1)|\int_0^1 r^{\beta Q-1}\,\mathrm{d}r\\
&=\left(\dfrac{-\alpha }{\beta Q}\right) \dfrac{|B(0,1)|}{\beta},
\end{align*}
which implies that
\begin{align}\label{H.Eq.4.6}
\exp{\left(\dfrac{\alpha}{\beta Q}\right)} \exp{\left(\dfrac{\beta}{|B(0, 1)|}\int_{B(0, 1)}|\xi|^{Q(\beta-1)}\ln (|\xi|^\alpha) \mathrm{d}\xi\right)}=1.
\end{align}
We apply \eqref{H.Eq.4.5} and the identity \eqref{H.Eq.4.6} in the left hand side of \eqref{H.Eq.4.4} to find that 
\begin{align}\label{H.F.D.Equ.3.9}
\begin{split}
\left(\int_{\mathbb{G}} \left[\exp{\left(\dfrac{\beta}{|B(0, |x|)|^{\beta}} \int_{B(0, |x|)}|B(0, |y|)|^{\beta-1}\ln g(y)\mathrm{d}y\right)}\right]^\frac{q}{p} w(x)\mathrm{d}x\right)^{\frac{1}{q}}=\\
\exp{\left(\dfrac{\alpha}{p \beta Q}\right)}\left(\int_{\mathbb{G}} \left[\exp{\left(\dfrac{\beta}{|B(0, 1)|}\int_{B(0, 1)} |\xi|^{Q(\beta-1)}\ln\left(|\xi|^\alpha g(|x|\xi)\right) \mathrm{d}\xi\right)}\right]^\frac{q}{p} w(x)\mathrm{d}x\right)^{\frac{1}{q}}.
\end{split}
\end{align}
Since, 
\begin{align*}
\dfrac{\beta}{|B(0, 1)|}\int_{B(0, 1)} |\xi|^{Q(\beta-1)} \mathrm{d}\xi=1,
\end{align*}
then by Jensen's inequality
\begin{align*}
\exp{\left(\dfrac{\beta}{|B(0, 1)|}\int_{B(0, 1)} |\xi|^{Q(\beta-1)}\ln\left(|\xi|^\alpha g(|x|\xi)\right) \mathrm{d}\xi\right)}\leq \\\dfrac{\beta}{|B(0, 1)|}\int_{B(0, 1)} |\xi|^{Q(\beta-1)+\alpha} g(|x|\xi)\, \mathrm{d}\xi,
\end{align*}
which implies that
\begin{align}\label{F.H.Thm.3.1.Equ.3.10}
\begin{split}
\exp{\left(\dfrac{\alpha}{p\beta Q}\right)}\left(\int_{\mathbb{G}} \left[\exp{\left(\dfrac{\beta}{|B(0, 1)|}\int_{B(0, 1)} |\xi|^{Q(\beta-1)}\ln\left(|\xi|^\alpha g(|x|\xi)\right) \mathrm{d}\xi\right)}\right]^\frac{q}{p} w(x)\mathrm{d}x\right)^\frac{1}{q}\\
\leq \left(\dfrac{\beta}{|B(0, 1)|}\exp{\left(\dfrac{\alpha}{\beta Q}\right)}\right)^{1/p}\left(\int_\mathbb{G} w(x)\left(\int_{B(0, 1)} |\xi|^{Q(\beta-1)+\alpha} g(|x|\xi) \mathrm{d}\,\xi\right)^\frac{q}{p}\,\mathrm{d}x\right)^\frac{1}{q}.
\end{split}
\end{align}
Let us change $|x|\xi$ by $\xi$ in \eqref{F.H.Thm.3.1.Equ.3.10}. Then, we have
\begin{align}
\begin{split}
\exp{\left(\dfrac{\alpha}{p\beta Q}\right)}\left(\int_{\mathbb{G}} \left[\exp{\left(\dfrac{\beta}{|B(0, 1)|}\int_{B(0, 1)} |\xi|^{Q(\beta-1)}\ln\left(|\xi|^\alpha g(|x|\xi)\right) \mathrm{d}\xi\right)}\right]^\frac{q}{p} w(x)\mathrm{d}x\right)^\frac{1}{q}\\
\leq \left(\dfrac{\beta \exp{\left(\dfrac{\alpha}{\beta Q}\right)}}{|B(0, 1)|}\right)^{1/p}\left(\int_\mathbb{G} |x|^{-(Q\beta+\alpha)\frac{q}{p}} w(x)\left(\int_{B(0, |x|)} |\xi|^{Q(\beta-1)+\alpha} g(\xi) \mathrm{d}\,\xi\right)^\frac{q}{p}\,\mathrm{d}x\right)^\frac{1}{q}.
\end{split}
\end{align}
Therefore, by using Minkowski's integral inequality when $p<q$ and Fubini's theorem when $p=q$, the later expression is less than or equal to
\begin{align}\label{H.Eq.4.8}
\begin{split}
\left(\dfrac{\beta \exp{\left(\dfrac{\alpha}{\beta Q}\right)}}{|B(0, 1)|}\right)^{1/p}  \left(\int_{\mathbb{G}}  g(\xi)\cdot |\xi|^{Q(\beta-1)+\alpha}\left(\int_{\mathbb{G}\setminus B(0, |\xi|)}   |x|^{-(Q\beta+\alpha)\frac{q}{p}} w(x)\mathrm{d}x\right)^\frac{p}{q} \mathrm{d}\xi\right)^{\frac{1}{p}}\\
\leq \left(\dfrac{\beta}{|B(0, 1)|}\exp{\left(\dfrac{\alpha}{\beta Q}\right)}\right)^{1/p} A(\alpha)\left(\int_{\mathbb{G}} g(x)\mathrm{d}x\right)^\frac{1}{p},
\end{split}
\end{align}
so that \eqref{H.Eq.4.4} follows from \eqref{H.Eq.4.2} and \eqref{H.F.D.Equ.3.9} -- \eqref{H.Eq.4.8}. Since \eqref{H.Eq.4.1} is equivalent to \eqref{H.Eq.4.4}, we conclude that \eqref{H.Eq.4.1} holds and the best constant $C$ satisfies
\begin{align*}
C\leq \left(\dfrac{\beta}{|B(0, 1)|}\right)^{1/p}\inf\limits_{\alpha>0}  \exp{\left(\dfrac{\alpha}{p\beta Q}\right)} A(\alpha).
\end{align*}

{\it Necessity.} To prove that \eqref{H.Eq.4.1}, or equivalently \eqref{H.Eq.4.4}, implies \eqref{H.Eq.4.2}, we define the test function $g$ on $\mathbb{G}$ by
\begin{align*}
g(x)=R^{-Q}\chi_{[\,0, R\, ]}(|x|)+\exp{\left(-\dfrac{Q\beta+\alpha}{\beta Q}\right)}\,\dfrac{R^{Q(\beta-1)+\alpha}}{|x|^{\beta Q+\alpha}}\, \chi_{(R, \infty)}(|x|),
\end{align*}
for fixed $R>0$, where $\alpha>Q(1-\beta)$.
Then, the right hand side of \eqref{H.Eq.4.4} becomes
\begin{align}\label{H.Eq.4.9}
\left(\int_{\mathbb{G}} g(x)\mathrm{d}x\right)^{1/p}=\left(\int_0^\infty \int_{\mathfrak{S}} r^{Q-1}g(ry)\mathrm{d}\sigma(y)\mathrm{d}r\right)^{1/p}\nonumber\\
=|\mathfrak{S}|^{1/p}\left(\int_0^R r^{Q-1} R^{-Q}\mathrm{d}r+ \exp{\left(-\frac{Q\beta+\alpha}{Q\beta}\right)}R^{Q(\beta-1)+\alpha}\int_R^\infty r^{Q(1-\beta)-\alpha-1}\mathrm{d}r\right)^{1/p}\nonumber\\
=|\mathfrak{S}|^{1/p}\left(\dfrac{1}{Q}+\dfrac{\exp{\left(-\frac{Q\beta+\alpha}{Q\beta}\right)}}{\alpha+Q(\beta-1)}\right)^{1/p}, \text{ since } \alpha>Q(1-\beta)\nonumber \\
=\left(\dfrac{|\mathfrak{S}|}{Q}\right)^{1/p}\left(1+ \dfrac{\exp{\left(-\frac{Q\beta+\alpha}{Q\beta}\right)}}{\alpha/Q+(\beta-1)}\right)^{1/p}
=|B(0, 1)|^{1/p}\left(1+\dfrac{\exp{\left(-\frac{Q\beta+\alpha}{Q\beta}\right)}}{\alpha/Q+(\beta-1)} \right)^{1/p}.
\end{align}
On the other hand, for $R>0$, we have that
\begin{align}\label{H.Eq.4.10}
\begin{split}
\left(\int_{\mathbb{G}\setminus B(0, R)} \left[\exp{\left(\dfrac{\beta}{|B(0, |x|)|^{\beta}} \int_{B(0, |x|)}|B(0, |y|)|^{\beta-1}\ln g(y)\mathrm{d}y\right)}\right]^\frac{q}{p} w(x)\mathrm{d}x\right)^{\frac{1}{q}}\\
\leq \left(\int_{\mathbb{G}} \left[\exp{\left(\dfrac{\beta}{|B(0, |x|)|^{\beta}} \int_{B(0, |x|)}|B(0, |y|)|^{\beta-1}\ln g(y)\mathrm{d}y\right)}\right]^\frac{q}{p} w(x)\mathrm{d}x\right)^{\frac{1}{q}}.
\end{split}
\end{align}
Moreover, for $0<R<|x|$, we have
\begin{align*}
\int_{B(0, |x|)}|B(0, |y|)|^{\beta-1}\ln g(y)\mathrm{d}y=\\
\left(\dfrac{|\mathfrak{S}|}{Q}\right)^{\beta-1}\int_0^{|x|}\int_{\mathfrak{S}} r^{Q\beta-1}\ln g(ry)\mathrm{d}\sigma(y)\mathrm{d}r\\
=Q\left(\dfrac{|\mathfrak{S}|}{Q}\right)^{\beta}\left[\int_0^R r^{Q\beta-1}\ln R^{-Q}\mathrm{d}r+\int_R^{|x|} r^{Q\beta-1}\ln \left(e^{-(\frac{Q\beta+\alpha}{\beta Q})}\right)\mathrm{d}r\right]\\
+Q\left(\dfrac{|\mathfrak{S}|}{Q}\right)^{\beta}\left[\int_R^{|x|} r^{Q\beta-1}\ln R^{Q(\beta-1)+\alpha}\mathrm{d}r-\int_R^{|x|} r^{Q\beta-1}\ln r^{\beta Q+\alpha}\mathrm{d}r\right]\\
=Q\left(\dfrac{|\mathfrak{S}|}{Q}\right)^{\beta}\Big\{\dfrac{1}{Q\beta}R^{Q\beta}\ln R^{-Q}-\left(\dfrac{Q\beta+\alpha}{(Q\beta)^2}\right)\Big(|x|^{Q\beta}-R^{Q\beta}\Big)\\+\dfrac{1}{Q\beta}\ln R^{Q(\beta-1)+\alpha}\Big(|x|^{Q\beta}-R^{Q\beta}\Big)\Big\}\\
-Q\left(\dfrac{|\mathfrak{S}|}{Q}\right)^{\beta}\left(\dfrac{Q\beta+\alpha}{Q\beta}\right)\left[\left(|x|^{Q\beta}\ln |x|-\dfrac{|x|^{Q\beta}}{Q\beta}\right)-\left(R^{Q\beta}\ln R-\dfrac{R^{Q\beta}}{Q\beta}\right)\right]\\
=Q \left(\dfrac{|\mathfrak{S}|}{Q}\right)^{\beta}\left[\dfrac{Q(\beta-1)+\alpha}{Q\beta}|x|^{Q\beta}\ln R-\dfrac{(Q\beta+\alpha)}{Q\beta}|x|^{Q\beta}\ln |x|\right]\\
=\dfrac{|B(0, |x|)|^\beta}{\beta}\ln \left(R^{Q(\beta-1)+\alpha} |x|^{-(\beta Q+\alpha)}\right),
\end{align*}
which implies that
\begin{align}\label{H.Eq.4.11}
\dfrac{\beta}{|B(0, |x|)|^{\beta}}\int_{B(0, |x|)}|B(0, |y|)|^{\beta-1}\ln g(y)\mathrm{d}y= \ln \left(R^{Q(\beta-1)+\alpha} |x|^{-(\beta Q+\alpha)}\right).
\end{align}
It follows from \eqref{H.Eq.4.4} and \eqref{H.Eq.4.9} -- \eqref{H.Eq.4.11} that
\begin{align*}
R^\frac{Q(\beta-1)+\alpha}{p} \left(\int_{\mathbb{G}\setminus B(0, R)}   |x|^{-(\beta Q+\alpha)\frac{q}{p}} w(x)\mathrm{d}x\right)^\frac{1}{q}\\
=\left(\int_{\mathbb{G}\setminus B(0, R)} \left[\exp{\left(\dfrac{\beta}{|B(0, |x|)|^{\beta}} \int_{B(0, |x|)}|B(0, |y|)|^{\beta-1}\ln g(y)\mathrm{d}y\right)}\right]^\frac{q}{p} w(x)\mathrm{d}x\right)^{\frac{1}{q}}\\
\leq C|B(0, 1)|^{1/p}\left(1+\dfrac{\exp{\left(-\dfrac{\beta Q+\alpha}{\beta Q}\right)}}{\alpha/Q+(\beta-1)} \right)^{1/p}.
\end{align*}
Therefore, 
\begin{align*}
A(\alpha)&=\sup\limits_{R>0}\,  R^\frac{{Q(\beta-1)+\alpha}}{p}\left(\int_{\mathbb{G}\setminus B(0, R)}   |x|^{-(Q\beta+\alpha)\frac{q}{p}} w(x)\mathrm{d}x\right)^\frac{1}{q}\\
 &\leq C|B(0, 1)|^{1/p}\left(1+\dfrac{\exp{\left(-\frac{\beta Q+\alpha}{\beta Q}\right)}}{\alpha/Q+(\beta-1)} \right)^{1/p}<\infty.
\end{align*}
We conclude that \eqref{H.Eq.4.2} holds and that the sharp constant $C$ satisfies \eqref{H.Eq.4.3}. The last statement of Theorem \ref{H.T.4.1} follows from the observation that the property \eqref{H.Eq.4.1} is independent of $\alpha$.
The proof is complete.
\end{proof}
\begin{remark}
A similar characterization of Theorem \ref{H.T.4.1} was proved in \cite[Theorem 3.1]{MRASBT}, but the condition is weaker than the one in Theorem \ref{H.T.4.1}. Moreover, we applied a direct method instead of a limit procedure. 
\end{remark}

First we note that Theorem \ref{H.T.4.1} implies the following inequality for the power weights $u(x)=|B(0, |x|)|^a$ and $v(x)=|B(0, |x|)|^b$:
\begin{example}\label{H.Exa.4.1}
Let $\mathbb{G}$ be a homogeneous group with the homogeneous dimension $Q$ equipped with a quasi-norm $|\cdot|$. Let $0<p\leq q<\infty$, $\beta>0$, and $a, b\in \mathbb{R}$. Then the inequality 
\begin{align}\label{H.Eq.4.12}
\begin{split}
    \left(\int_{\mathbb{G}} \left[\exp{\left(\dfrac{\beta}{|B(0, |x|)|^{\beta}} \int_{B(0, |x|)}|B(0, |y|)|^{\beta-1}\ln f(y)\,\mathrm{d}y\right)}\right]^q |B(0, |x|)|^a\,\mathrm{d}x\right)^{\frac{1}{q}}\\
    \leq C \left(\int_{\mathbb{G}} f^p(x)|B(0, |x|)|^b\,\mathrm{d}x\right)^{\frac{1}{p}}
    \end{split}
\end{align}
holds for all positive functions $f$ if and only if
\begin{align}\label{H.Eq.4.13}
\dfrac{1+a}{q}=\dfrac{1+b}{p}.
\end{align}
Moreover, the best constant $C$ in \eqref{H.Eq.4.12} satisfies 
\begin{align}\label{H.Eq.4.14}
\begin{split}
\left(\dfrac{p}{q}\right)^{1/q}\exp{\left(\frac{b}{p\beta}\right)}\sup\limits_{\alpha>0} \left(1+\dfrac{\exp{\left(-\frac{Q\beta+\alpha}{Q\beta}\right)}}{\alpha/Q+(\beta-1)} \right)^{-1/p}\left(\frac{\alpha}{Q}+\beta-1\right)^{-1/q}\\
\leq C\leq \beta^{\frac{1}{p}-\frac{1}{q}}\exp{\left(\dfrac{1}{q}+\dfrac{1-\beta+b}{p\beta}\right)}.
\end{split}
\end{align}
Indeed, in this case \eqref{H.Eq.4.2.1} and \eqref{H.Eq.4.2} are of the form
\begin{align*}
w(x)=\exp{\left(\dfrac{bq}{\beta p}\right)}|B(0,|x|)|^{a-b\frac{q}{p}}
\end{align*}
and
\begin{align*}
A(\alpha)=\left(\dfrac{p}{q}\right)^{1/q}|B(0, 1)|^{\frac{1+a}{q}-\frac{b}{p}}\dfrac{\exp{\left(\frac{b}{p\beta}\right)}\sup\limits_{R>0} R^{Q\left(\frac{1+a}{q}-\frac{1+b}{p}\right)}}{[(b+\beta)+\alpha/Q-(1+a)\frac{p}{q}]^{1/q}} ,
\end{align*}
provided that $\alpha>Qp\left(\dfrac{1+a}{q}-\dfrac{\beta+b}{p}\right)$. We observe that the supremum in $R>0$ is finite if and only if \eqref{H.Eq.4.13} holds. Furthermore, if condition \eqref{H.Eq.4.13} is  satisfied, then we have
\begin{align}\label{H.Eq.4.15}
A(\alpha)=\left(\dfrac{p}{q}\right)^{1/q}|B(0, 1)|^{1/p}\dfrac{\exp{\left(\frac{b}{p\beta}\right)}}{[\alpha/Q+(\beta-1)]^{1/q}}.
\end{align}
Consequently, from \eqref{H.Eq.4.3} and \eqref{H.Eq.4.15}, we have
\begin{align}\label{H.Eq.4.16}
\begin{split}
\left(\dfrac{p}{q}\right)^{1/q}\exp{\left(\frac{b}{p\beta}\right)}\sup\limits_{\alpha>0} \left(1+\dfrac{\exp{\left(-\dfrac{Q\beta+\alpha}{Q\beta}\right)}}{\alpha/Q+(\beta-1)} \right)^{-1/p}\left(\dfrac{\alpha}{Q}+(\beta-1)\right)^{-1/q}\\
\leq C
\leq \left(\dfrac{p}{q}\right)^{1/q} \beta^{\frac{1}{p}-\frac{1}{q}} \exp{\left(\frac{b}{p\beta}\right)}\left[\inf\limits_{\alpha>0}   \dfrac{\exp{\left(\frac{q\alpha}{pQ\beta}\right)}}{\left(\frac{\alpha}{Q\beta}+1-\frac{1}{\beta}\right)}\right]^\frac{1}{q}.
\end{split}
\end{align}
Moreover, 
\begin{align}\label{H.Eq.4.17}
\left[\inf\limits_{\alpha>0}   \dfrac{\exp{\left(\frac{q\alpha}{pQ\beta}\right)}}{\left(\frac{\alpha}{Q\beta}+1-\frac{1}{\beta}\right)}\right]^{1/q}=\left(\dfrac{q}{p}\right)^{1/q} \exp{\left(\frac{1}{q}+\dfrac{1}{p\beta}-\dfrac{1}{p}\right)}.
\end{align}
Therefore, \eqref{H.Eq.4.14} follows from \eqref{H.Eq.4.16} and \eqref{H.Eq.4.17}.
\end{example}
\begin{remark}
Let $0<p\leq q<\infty$. If the condition \eqref{H.Eq.4.13} in Example \ref{H.Exa.4.1} holds, then the sharp constant $C$ in \eqref{H.Eq.4.12} satisfies
\begin{align*}
C\leq \beta^{\frac{1}{p}-\frac{1}{q}}\exp{\left(\dfrac{1}{q}-\dfrac{1}{p}\right)}\exp{\left(\dfrac{1+b}{\beta p}\right)}.\\
\end{align*}
\end{remark}
Moreover, we state and prove the following inequality for the power weights $u(x)=v(x)=|B(0, |x|)|^a$ with $p=q=1$:
\begin{theorem}\label{H.T.4.2}
Let $\mathbb{G}$ be a homogeneous group with the homogeneous dimension $Q$ equipped with a quasi-norm $|\cdot|$, and let $\beta>0$ and $a\in \mathbb{R}$. Then the inequality 
\begin{align}\label{H.Eq.4.19}
\begin{split}
    \int_{\mathbb{G}} \exp{\left(\dfrac{\beta}{|B(0, |x|)|^{\beta}} \int_{B(0, |x|)}|B(0, |y|)|^{\beta-1}\ln f(y)\mathrm{d}y\right)} |B(0, |x|)|^a\,\mathrm{d}x\\
    \leq e^{(a+1)/\beta} \int_{\mathbb{G}} f(x)|B(0, |x|)|^a\,\mathrm{d}x
    \end{split}
\end{align}
holds for all positive functions $f$ on $\mathbb{G}$, and the constant $e^{(a+1)/\beta}$ in \eqref{H.Eq.4.19} is sharp. 
\end{theorem}
\begin{proof}
In view of Example \ref{H.Exa.4.1}, the inequality \eqref{H.Eq.4.19} holds with some constant. Now, we need to prove that the sharp constant is $C=e^{(a+1)/\beta}$. From \eqref{H.Eq.4.2} and \eqref{H.Eq.4.2.1} with $p=q=1$, we have that
\begin{align*}
w(x)=\exp{\left(\dfrac{a}{\beta}\right)} \text{ and } A(\alpha)=\exp{\left(\dfrac{a}{\beta}\right)}\dfrac{|B(0,1)|}{\alpha/Q+(\beta-1)}.
\end{align*}

It follows from \eqref{H.Eq.4.3} that the sharp constant $C$ satisfies
\begin{align*}
C\leq \beta \exp{\left(\dfrac{a}{\beta}\right)}\inf\limits_{\alpha>0} \dfrac{\exp{\left(\dfrac{\alpha}{\beta  Q}\right)}}{(\beta-1)+\alpha/Q}. 
\end{align*}
The infimum in the above inequality is attained at $\alpha=Q$. Hence, 
\begin{align}\label{H.Eq.4.20}
C\leq e^{(1+a)/\beta}.
\end{align}
It only remains to prove that the inequality \eqref{H.Eq.4.20} also holds in the reverse direction.  Consider the function
\begin{align*}
f(x)= \chi_{[ 0,\,e^{\frac{1}{\beta Q}} ]}(|x|)+|x|^{-\gamma}\chi_{(e^{\frac{1}{\beta Q}},\,\infty)}(|x|),
\end{align*}
where $\gamma>Q(a+1)$. Then, the integral part of the right hand side of \eqref{H.Eq.4.19} becomes
\begin{align}\label{H.Eq.4.21}
\int_{\mathbb{G}} f(x)|B(0, |x|)|^a\,\mathrm{d}x&= \int_0^\infty \int_{\mathfrak{S}} r^{Q-1} f(ry)|B(0,r)|^a\, \mathrm{d}\sigma(y)\mathrm{d}r\nonumber\\
&=\dfrac{|\mathfrak{S}|^{a+1}}{Q^a} \int_0^\infty r^{Q(a+1)-1} \Big(\chi_{[ 0,e^{\frac{1}{\beta Q}} ]}(r)+r^{-\gamma}\chi_{(e^{\frac{1}{\beta Q}},\infty)}(r)\Big)\, \mathrm{d}r\nonumber\\
&= |B(0, 1)|^{a+1} e^{(a+1)/\beta}\left(\dfrac{1}{(a+1)}+\dfrac{e^{-\gamma/\beta Q}}{\gamma/Q-(a+1)}\right),
\end{align}
and the left hand side of \eqref{H.Eq.4.19} becomes
\begin{align}\label{H.Eq.4.22}
\int_{\mathbb{G}} \exp{\left(\dfrac{\beta}{|B(0, |x|)|^{\beta}}\int_{B(0, |x|)} |B(0, |y|)|^{\beta-1}\ln f(y)\,\mathrm{d}y\right)} |B(0, |x|)|^a\,\mathrm{d}x\nonumber=\\
 \int_0^\infty \int_{\mathfrak{S}} r^{Q-1} \exp{\left(\dfrac{\beta}{|B(0,r)|^\beta}\int_{B(0,r)} |B(0, |y|)|^{\beta-1}\ln f(y)\,\mathrm{d}y\right)} |B(0,r)|^a\,\mathrm{d}\sigma(y)\mathrm{d}r=\nonumber\\
\dfrac{|\mathfrak{S}|^{a+1}}{Q^a}\int_0^\infty r^{Q(a+1)-1} \exp{\left(\dfrac{\beta Q}{r^{\beta Q}}\int_0^r s^{\beta Q-1}\ln \Big(\chi_{[ 0,e^{\frac{1}{\beta Q}}]}(s)+s^{-\gamma}\chi_{(e^{\frac{1}{\beta Q}},\infty)}(s)\Big)\mathrm{d}s\right)}\mathrm{d}r\nonumber\\
= |B(0, 1)|^{a+1} e^{(a+1)/\beta}\left(\dfrac{1}{(a+1)}+\dfrac{1}{\gamma/Q-(a+1)}\right) .
\end{align}
It follows from \eqref{H.Eq.4.19}, \eqref{H.Eq.4.21} and \eqref{H.Eq.4.22} that
\begin{align*}
\dfrac{e^{\gamma/\beta Q}}{\frac{Q}{\gamma}(a+1)+\left(1-\frac{Q}{\gamma}(a+1)\right)e^{\gamma/\beta Q}}\leq C.
\end{align*} 
By letting $\gamma/Q\to (a+1)^+$, we find that 
\begin{align}\label{H.Eq.4.23}
e^{(a+1)/\beta}\leq C.
\end{align}
Therefore, the sharpness of the constant in \eqref{H.Eq.4.19} follows by combining \eqref{H.Eq.4.20} and \eqref{H.Eq.4.23}. The proof is complete.
\end{proof}

Next, by using a suitable transformations and arguments as those in the proof of Theorem \ref{H.T.4.1}, we can also derive the following dual version.
\begin{theorem}\label{H.T.5.1}
Let $\mathbb{G}$ be a homogeneous group with the homogeneous dimension $Q$ equipped with a quasi-norm $|\cdot|$, and let $0<p\leq q<\infty$ and $\beta>0$. Suppose that $u$ and $v$ are positive functions on $\mathbb{G}$. Then 
\begin{align}\label{H.Eq.5.1}
\begin{split}
    \left(\int_{\mathbb{G}} \left[\exp{\left(\beta |B(0, |x|)|^{\beta} \int_{\mathbb{G}\setminus B(0, |x|)}|B(0, |y|)|^{-(\beta+1)}\ln f(y)\mathrm{d}y\right)}\right]^q u(x)\mathrm{d}x\right)^{\frac{1}{q}}\\
    \leq C \left(\int_{\mathbb{G}} f^p(x)v(x)\mathrm{d}x\right)^{\frac{1}{p}}
    \end{split}
\end{align}
holds for some finite $C$ and for all positive functions $f$ if and only if for some $\alpha>0$
\begin{align}\label{H.Eq.5.2}
\widetilde{A}(\alpha):=\sup\limits_{R>0}\,  R^\frac{{Q(\beta-1)+\alpha}}{p}\left(\int_{\mathbb{G}\setminus B(0, R)}   |x|^{-(Q\beta+\alpha)\frac{q}{p}} \widetilde{w}(x)\mathrm{d}x\right)^\frac{1}{q}<\infty,
\end{align}
where 
\begin{align}\label{H.Eq.5.2.1}
\widetilde{w}(x)=\widetilde{u}(x)\left[\exp{\left(\dfrac{\beta}{|B(0, |x|)|^{\beta}}\int_{B(0, |x|)}|B(0, |y|)|^{\beta-1}\ln \dfrac{1}{\widetilde{v}(y)} \mathrm{d}y\right)}\right]^\frac{q}{p},
\end{align}
and
\begin{align}\label{H.Eq.5.2.2}
\widetilde{u}(ry)=|ry|^{-2Q}u\left(\frac{1}{r}y\right) \text{ and }\widetilde{v}(ry)=|ry|^{-2Q}v\left(\frac{1}{r}y\right)
\end{align}
 for $r>0$. Moreover, the best constant $C$ in \eqref{H.Eq.5.1} satisfies
\begin{align}\label{H.Eq.5.3}
\begin{split}
\left(\dfrac{1}{|B(0, 1)|}\right)^{1/p}\sup\limits_{\alpha>0} \left(1+\dfrac{\exp{\left(-\dfrac{\beta Q+\alpha}{\beta Q}\right)}}{\alpha/Q+(\beta-1)} \right)^{-1/p}  \widetilde{A}(\alpha)
\leq C\\\leq \left(\dfrac{\beta}{|B(0, 1)|}\right)^{1/p}\inf\limits_{\alpha>0}  \exp{\left(\dfrac{\alpha}{\beta p Q}\right)} \widetilde{A}(\alpha). 
\end{split}
\end{align}
Furthermore, if $\widetilde{A}(\alpha)<\infty$ for some $\alpha>0$, then $\widetilde{A}(\alpha)<\infty$ for all $\alpha>0$.
\end{theorem}
\begin{proof}
By making variable transformations we have that
\begin{align}\label{H.Eq.5.4}
\begin{split}
\int_{\mathbb{G}} \left[\exp{\left(\beta |B(0, |x|)|^{\beta} \int_{\mathbb{G}\setminus B(0, |x|)}|B(0, |y|)|^{-(\beta+1)}\ln f(y)\mathrm{d}y\right)}\right]^q u(x)\mathrm{d}x\\
=\int_{\mathbb{G}} \left[\exp{\left(\dfrac{\beta}{|B(0, |x|)|^{\beta}}\int_{B(0, |x|)}|B(0, |y|)|^{\beta-1}\ln g(y)\mathrm{d}y\right)}\right]^q \widetilde{u}(x)\mathrm{d}x
\end{split}
\end{align}
and 
\begin{align}\label{H.Eq.5.5}
\int_{\mathbb{G}} f^p(x)v(x)\mathrm{d}x=\int_{\mathbb{G}} g^p(x)\widetilde{v}(x)\mathrm{d}x,
\end{align}
where $g(ry)=f\left(\frac{1}{r}y\right)$ for $r>0$ with $\widetilde{u}(x)$ and $\widetilde{v}(x)$ are defined by \eqref{H.Eq.5.2.2}. It follows from \eqref{H.Eq.5.4} and \eqref{H.Eq.5.5} that, the inequality \eqref{H.Eq.5.1} is equivalent to 
\begin{align}\label{H.Eq.5.6}
\begin{split}
\left(\int_{\mathbb{G}} \left[\exp{\left(\dfrac{\beta}{|B(0, |x|)|^{\beta}} \int_{B(0, |x|)}|B(0, |y|)|^{\beta-1}\ln g(y)\mathrm{d}y\right)}\right]^q \widetilde{u}(x)\mathrm{d}x\right)^{\frac{1}{q}}\\
\leq C \left(\int_{\mathbb{G}} g^p(x)\widetilde{v}(x)\mathrm{d}x\right)^{\frac{1}{p}},
\end{split}
\end{align}
where $\widetilde{u}(x)$ and $\widetilde{v}(x)$ are defined by \eqref{H.Eq.5.2.2}.

In view of Theorem \ref{H.T.4.1}, the inequality \eqref{H.Eq.5.6} holds for some finite $C$ if and only if for some $\alpha>0$,
\begin{align*}
\widetilde{A}(\alpha)=\sup\limits_{R>0}\,  R^\frac{{Q(\beta-1)+\alpha}}{p}\left(\int_{\mathbb{G}\setminus B(0, R)}   |x|^{-(Q\beta+\alpha)\frac{q}{p}} \widetilde{w}(x)\mathrm{d}x\right)^\frac{1}{q}<\infty,
\end{align*}
where $\widetilde{w}(x)$ is defined by \eqref{H.Eq.5.2.1} with $\widetilde{u}(x)$ and $\widetilde{v}(x)$ are defined by \eqref{H.Eq.5.2.2}. Since the inequality \eqref{H.Eq.5.1} is equivalent to \eqref{H.Eq.5.6} with $\widetilde{u}(x)$ and $\widetilde{v}(x)$ are defined by \eqref{H.Eq.5.2.2}, we can, by Theorem \ref{H.T.4.1}, conclude that \eqref{H.Eq.5.1} holds if and only if \eqref{H.Eq.5.2} - \eqref{H.Eq.5.2.2} holds.
Moreover, the sharp constant $C$ in \eqref{H.Eq.5.1} satisfies the estimates \eqref{H.Eq.5.3}. The proof is complete.
\end{proof}
Now let us prove the dual version of Levin-Cochran-Lee's type inequality in Theorem \ref{H.T.4.2}:
\begin{theorem}\label{H.T.5.2}
Let $\mathbb{G}$ be a homogeneous group with the homogeneous dimension $Q$ equipped with a quasi-norm $|\cdot|$, and let $a\in \mathbb{R}$ and $\beta>0$. Suppose that $u$ and $v$ are positive functions on $\mathbb{G}$. Then 
\begin{align}\label{H.Eq.5.9}
\begin{split}
    \int_{\mathbb{G}} \exp{\left(\beta |B(0, |x|)|^{\beta} \int_{\mathbb{G}\setminus B(0, |x|)}|B(0, |y|)|^{-(\beta+1)}\ln f(y)\mathrm{d}y\right)} |B(0, |x|)|^a\mathrm{d}x\\
    \leq e^{-(a+1)/\beta} \int_{\mathbb{G}} f(x)|B(0, |x|)|^a\mathrm{d}x
    \end{split}
\end{align}
holds for all positive functions $f$, and the constant $e^{-(a+1)/\beta}$ in \eqref{H.Eq.5.9} is sharp.
\end{theorem}
\begin{proof}
Using similar approach as we did in Theorem \ref{H.T.5.1}, we can show that \eqref{H.Eq.5.9} is equivalent to  
\begin{align*}
\begin{split}
\int_{\mathbb{G}} \exp{\left(\dfrac{\beta}{|B(0, |x|)|^{\beta}}  \int_{B(0, |x|)}|B(0, |y|)|^{\beta-1}\ln g(y)\mathrm{d}y\right)} |B(0, |x|)|^{-(a+2)}\,\mathrm{d}x\\
\leq e^{-(a+1)/\beta}\int_{\mathbb{G}} g(y) |B(0, |x|)|^{-(a+2)}\,\mathrm{d}x.
\end{split}
\end{align*}
The next steps are similar to the steps in the proof of Theorem \ref{H.T.4.2}, so we omit the details. 
\end{proof}
\begin{corollary}\label{H.C.3.4}
Let $\mathbb{G}$ be a homogeneous group with the homogeneous dimension $Q$ equipped with a quasi-norm $|\cdot|$, and let $\beta>0$, $a_i>0$ ($i=1,2,...,n$) and $k\in \mathbb{N}\cup\{0\}=:\mathbb{N}_0$. Then the inequality 
\begin{align}\label{H.Eq.4.23.1}
\begin{split}
    \int_{\mathbb{G}} \exp{\left(\dfrac{\beta}{|B(0, |x|)|^{\beta}} \int_{B(0, |x|)}|B(0, |y|)|^{\beta-1}\ln f(y)\mathrm{d}y\right)} \left(\sum\limits_{i=1}^n |B(0, |x|)|^{a_i}\right)^k\,\mathrm{d}x\\
    \leq \exp{\left(\dfrac{1+k(a_1+\cdots+a_n)}{\beta}\right)} \int_{\mathbb{G}} f(x)\left(\sum\limits_{i=1}^n |B(0, |x|)|^{a_i}\right)^k\,\mathrm{d}x,
    \end{split}
\end{align}
holds for all positive functions $f$. Concerning the sharp constant see our Proposition \ref{H.P.2}.
\end{corollary}
\begin{proof}
In view of Theorem \ref{H.T.4.2} and by applying the multinomial theorem, we have that
\begin{align}\label{H.Eq.4.23.2}
 \int_{\mathbb{G}} \exp{\left(\dfrac{\beta}{|B(0, |x|)|^{\beta}}  \int_{B(0, |x|)}|B(0, |y|)|^{\beta-1}\ln f(y)\mathrm{d}y\right)} \left(\sum\limits_{i=1}^n |B(0, |x|)|^{a_i}\right)^k\mathrm{d}x\nonumber\\
 =\sum\limits_{\substack{\sum\limits_{i=1}^n m_i=k\\m_i\in \mathbb{N}_0}}  \begin{pmatrix}
k\\
m_1,...,m_n
\end{pmatrix}\times\nonumber\\\int_{\mathbb{G}} \exp{\left(\dfrac{\beta}{|B(0, |x|)|^{\beta}} \int_{B(0, |x|)}|B(0, |y|)|^{\beta-1}\ln f(y)\mathrm{d}y\right)}  |B(0, |x|)|^{\sum\limits_{i=1}^n a_im_i}\mathrm{d}x\nonumber\\
\leq \sum\limits_{\substack{\sum\limits_{i=1}^n m_i=k\\m_i\in \mathbb{N}_0}}  \begin{pmatrix}
k\\
m_1,...,m_n
\end{pmatrix} \exp{\left(\dfrac{1+a_1m_1+\cdots+a_nm_n}{\beta}\right)}\int_{\mathbb{G}} f(x) |B(0, |x|)|^{\sum\limits_{i=1}^n a_im_i}\mathrm{d}x.
\end{align}
Moreover,
\begin{align}\label{H.Eq.4.23.3}
\sum\limits_{i=1}^n a_im_i\leq k\left(\sum\limits_{i=1}^n a_i\right).
\end{align}
Therefore, \eqref{H.Eq.4.23.1} follows from \eqref{H.Eq.4.23.2}, \eqref{H.Eq.4.23.3}, and the multinomial theorem. 
\end{proof}

We conclude this section by pointing out the following inequality with different weight functions, which, in particular generalizes a result in \cite{YPA}:
\begin{corollary}\label{H.C.5.3}
Let $\mathbb{G}$ be a homogeneous group with the homogeneous dimension $Q$ equipped with a quasi-norm $|\cdot|$, and let $\beta>0$. Suppose that $u$ and $v$ are positive functions on $\mathbb{G}$. Then the inequality 
\begin{align}\label{H.Eq.4.24}
    \int_{\mathbb{G}} \exp{\left(\dfrac{\beta}{|B(0, |x|)|^{\beta}} \int_{B(0, |x|)}|B(0, |y|)|^{\beta-1}\ln f(y)\mathrm{d}y\right)} u(x)\,\mathrm{d}x
    \leq e^{1/\beta} \int_{\mathbb{G}} f(x)v(x)\,\mathrm{d}x
\end{align} 
holds for all positive functions $f$ provided that
\begin{align}\label{H.Eq.4.25}
u(x)= \exp{\left(\dfrac{\beta}{|B(0, |x|)|^{\beta}} \int_{B(0, |x|)} |B(0, |y|)|^{\beta-1}\ln v(y)\mathrm{d}y\right)}.
\end{align}
Moreover, the constant $e^{1/\beta}$ in \eqref{H.Eq.4.24} is sharp.
\end{corollary}
\begin{proof}
Let $g(x)=f(x)v(x)$. Then, the inequality \eqref{H.Eq.4.24} is equivalent to the inequality
\begin{align}\label{H.Eq.4.26}
\int_{\mathbb{G}} \exp{\left(\dfrac{\beta}{|B(0, |x|)|^{\beta}} \int_{B(0, |x|)}|B(0, |y|)|^{\beta-1}\ln g(y)\mathrm{d}y\right)}\,\mathrm{d}x
    \leq e^{1/\beta} \int_{\mathbb{G}} g(x)\,\mathrm{d}x,
\end{align}
provided that \eqref{H.Eq.4.25} holds. In view of Theorem \ref{H.T.4.2} with $a=0$, we have that indeed the inequality \eqref{H.Eq.4.26} holds and the constant $e^{1/\beta}$ is sharp. Therefore, from the equivalence of \eqref{H.Eq.4.24} and \eqref{H.Eq.4.26}, we can conclude that \eqref{H.Eq.4.24} holds and the constant $e^{1/\beta}$ in \eqref{H.Eq.4.24} is sharp. The proof is complete.
\end{proof}
\section{Final Remarks}
Here, we give one example as an application of Corollary \ref{H.C.5.3} for another type of weight functions $u(x)=\exp{\left(\dfrac{\eta}{1+\gamma/\beta Q}|x|^{\gamma}\right)}$ and $v(x)=\exp{(\eta |x|^{\gamma})}$:
\begin{example}
Let $\mathbb{G}$ be a homogeneous group with the homogeneous dimension $Q$ equipped with a quasi-norm $|\cdot|$, and let $\beta>0$, $\eta, \gamma\in \mathbb{R}$. Suppose that $f$ is a positive function on $\mathbb{G}$. Then, the following inequality 
\begin{align*}
\begin{split}
    \int_{\mathbb{G}} \exp{\left(\dfrac{\beta}{|B(0, |x|)|^{\beta}} \int_{B(0, |x|)}|B(0, |y|)|^{\beta-1}\ln f(y)\mathrm{d}y\right)} \exp{\left(\dfrac{\eta}{1+\gamma/\beta Q}|x|^{\gamma}\right)}\,\mathrm{d}x\\
    \leq e^{1/\beta} \int_{\mathbb{G}} f(x)\exp{(\eta |x|^{\gamma})}\,\mathrm{d}x
    \end{split}
\end{align*} 
holds and the constant $e^{1/\beta}$ is sharp.
\end{example}

\begin{proposition}\label{H.P.1}
Let $\mathbb{G}$ be a homogeneous group with the homogeneous dimension $Q$ equipped with a quasi-norm $|\cdot|$, and let $0<p\leq q<\infty$ and $\beta>0$. Suppose that $u, v$ and $f$ are positive functions on $\mathbb{G}$. Then 
\begin{align}\label{H.Eq.6.1}
\begin{split}
    \left(\int_{\mathbb{G}} \left[\exp{\left(\dfrac{\beta}{|B(0, |x|)|^{\beta}} \int_{B(0, |x|)}|B(0, |y|)|^{\beta-1}\ln f(y)\mathrm{d}y\right)}\right]^q u(x)\mathrm{d}x\right)^{\frac{1}{q}}\\
    \leq C \left(\int_{\mathbb{G}} f^p(x)v(x)\mathrm{d}x\right)^{\frac{1}{p}}
    \end{split}
\end{align}
is equivalent to the inequality
\begin{align}\label{H.Eq.6.2}
    \left(\int_{\mathbb{G}} \left[\exp{\left(\dfrac{1}{|B(0, |x|)|}\int_{B(0, |x|)}\ln g(y)\mathrm{d}y\right)}\right]^q u_\beta(x)\mathrm{d}x\right)^{\frac{1}{q}}
    \leq C \left(\int_{\mathbb{G}} g^p(x)v_\beta(x)\mathrm{d}x\right)^{\frac{1}{p}},
\end{align}
for finite constant $C>0$, where
\begin{align}\label{H.Eq.6.3}
u_\beta(ry)=\dfrac{|ry|^{\frac{1-\beta}{\beta}}}{\beta}u(r^\frac{1}{\beta}y),\,\, v_\beta(ry)=\dfrac{|ry|^{\frac{1-\beta}{\beta}}}{\beta}v(r^\frac{1}{\beta}y)\, \text{ and } g(ry)=f(r^\frac{1}{\beta}y), 
\end{align}
for $r>0$. Moreover, the equivalence of \eqref{H.Eq.6.1} and \eqref{H.Eq.6.2} holds with the same constant $C>0$.
\end{proposition}
\begin{proof}
By making the variable transformations $s=r^\beta$, we find that the inequality \eqref{H.Eq.6.1} is equivalent to 
\begin{align}\label{H.Eq.6.4}
\begin{split}
\Big\{\int_o^\infty \int_{\mathfrak{S}} \left[\exp{\left(\dfrac{\beta Q^\beta}{|\mathfrak{S}|^\beta s^Q} \int_{B(0,s^\frac{1}{\beta})}|B(0, |y|)|^{\beta-1}\ln f(y)\mathrm{d}y\right)}\right]^q\times\\
\dfrac{s^{Q-1+\frac{1-\beta}{\beta}}}{\beta} u(s^{\frac{1}{\beta}}z)\mathrm{d}\sigma(z)\mathrm{d}s\Big\}^{\frac{1}{q}}\\
    \leq C \left(\int_0^\infty \int_{\mathfrak{S}} f^p(r^\frac{1}{\beta}y) \dfrac{r^{Q-1+\frac{1-\beta}{\beta}}}{\beta} v(r^{\frac{1}{\beta}}z)\mathrm{d}\sigma(z)\mathrm{d}r\right)^{\frac{1}{p}}.
    \end{split}
\end{align}
Moreover, by making similar variable transformations as above, we have
\begin{align}\label{H.Eq.6.5}
\int_{B(0,s^\frac{1}{\beta})}|B(0, |y|)|^{\beta-1}\ln f(y)\mathrm{d}y=\dfrac{1}{\beta}\left(\dfrac{|\mathfrak{S}|}{Q}\right)^{\beta-1}\int_{B(0,s)} \ln g(y)\,\mathrm{d}y,
\end{align}
where $g(y)$ is defined by \eqref{H.Eq.6.3}.

 Therefore, from \eqref{H.Eq.6.4} and \eqref{H.Eq.6.5}, we can conclude that the inequality \eqref{H.Eq.6.1} is equivalent to \eqref{H.Eq.6.2}, where $g, u_\beta$ and $v_\beta$ are defined by \eqref{H.Eq.6.3}. The proof is complete.
\end{proof}
\begin{proposition}\label{H.P.2}
The sharp constant $C$ in the inequality \eqref{H.Eq.4.23.1} satisfies
\begin{align}\label{H.Eq.6.6.1}
\exp{\left(\dfrac{1+ a_1m_1+\cdots+a_nm_n}{\beta}\right)}\leq C\leq \exp{\left(\dfrac{1+k(a_1+\cdots +a_n)}{\beta}\right)},
\end{align}
for at least one collection $\{m_1, ..., m_n\}\subset \mathbb{N}_0$ such that $m_1+\cdots+m_n=k$.
\end{proposition}
\begin{proof}
In view of Corollary \ref{H.C.3.4}, we have that
\begin{align}\label{H.Eq.6.7}
C\leq \exp{\left(\dfrac{1+k(a_1+\cdots +a_n)}{\beta}\right)}.
\end{align}
It only remains to prove that
\begin{align}\label{H.Eq.6.8}
\exp{\left(\dfrac{1+ a_1m_1+\cdots+a_nm_n}{\beta}\right)}\leq C.
\end{align}
Consider the function
\begin{align*}
f(x)= \chi_{[ 0,\,e^{\frac{1}{\beta Q}} ]}(|x|)+|x|^{-\gamma}\chi_{(e^{\frac{1}{\beta Q}},\,\infty)}(|x|),
\end{align*}
where $\gamma>Q\left(1+a\right)$ with $a:=\sum\limits_{i=1}^n m_ia_i$ for arbitrary  collection $\{m_1, ..., m_n\}\subset \mathbb{N}_0$ such that $k=\sum\limits_{i=1}^n m_i$. Then, from \eqref{H.Eq.4.23.1}, we have
\begin{align*}
\begin{split}
\sum\limits_{\substack{\sum\limits_{i=1}^n m_i=k\\m_i\in \mathbb{N}_0}}  \begin{pmatrix}
k\\
m_1,...,m_n
\end{pmatrix} |B(0, 1)|^{1+a}\exp{\left(\dfrac{1+a}{\beta}\right)}\left(\dfrac{1}{1+a}+\dfrac{1}{\gamma/Q-\left(1+a\right)}\right)\\
\leq C \sum\limits_{\substack{\sum\limits_{i=1}^n m_i=k\\m_i\in \mathbb{N}_0}}  \begin{pmatrix}
k\\
m_1,...,m_n
\end{pmatrix} |B(0, 1)|^{1+a}\exp{\left(\dfrac{1+a}{\beta}\right)}\left(\dfrac{1}{1+a}+\dfrac{e^{-\gamma/\beta Q}}{\gamma/Q-\left(1+a\right)}\right)<\infty.
\end{split}
\end{align*}
Consequently, there exists a collection $\{m_1, ..., m_n\}\subseteq \mathbb{N}_0$ with $\sum\limits_{i=1}^n m_i=k$ such that
\begin{align*}
\left(\dfrac{1}{1+a}+\dfrac{1}{\gamma/Q-\left(1+a\right)}\right)\leq C \left(\dfrac{1}{1+a}+\dfrac{e^{-\gamma/\beta Q}}{\gamma/Q-\left(1+a\right)}\right).
\end{align*}
Then, the above inequality implies that
\begin{align*}
\dfrac{e^{\gamma/\beta Q}}{\left(\frac{1+a}{\gamma/Q}\right)+\left(1-\frac{1+a}{\gamma/Q}\right)e^{\gamma/\beta Q}}\leq C.
\end{align*}
By letting $\gamma/Q\to (1+a)^+$, we showed that \eqref{H.Eq.6.8} holds. Therefore, the inequality \eqref{H.Eq.6.6.1} follows from \eqref{H.Eq.6.7} and \eqref{H.Eq.6.8}. The proof is complete.
\end{proof}

\subsubsection*{Acknowledgements}
The research in this paper was initiated and carried out with the support of Ghent Analysis \& PDE Centre at Ghent University, Belgium, when the second author came for a long-term research visit. The second author is very grateful to the centre for the support and warm hospitality during his research visit. 

\subsection*{Funding}
The first author was supported by the FWO Odysseus 1 grant no. G.0H94.18N: Analysis and Partial Differential Equations, by the Methusalem programme of the Ghent University Special Research Fund (BOF) (grant no. 01M01021) and by the EPSRC (grants no. EP/R003025/2 and EP/V005529/1).
\subsection*{Authors' contributions}
All authors contributed equally to this paper and approved the final manuscript.

\subsection*{Data availability}
Not applicable
\section*{Declarations}

\subsection*{Competing interests}
The authors declare that they have no competing interests.


\end{document}